\documentclass{m2an}
\newtheorem{theorem}{Theorem}[section]

\newtheorem{remark}{Remark}[section]

\newtheorem{example}{Example}[section]

\begin{document}
\title{Local discontinuous Galerkin method for the fractional diffusion equation with integral fractional Laplacian}\thanks{This work was supported by the National Natural Science Foundation of China under Grant No. 12071195, and the AI and Big Data Funds under Grant No. 2019620005000775.}
\author{Daxin Nie}\address{School of Mathematics and Statistics, Gansu Key Laboratory of Applied Mathematics and Complex Systems, Lanzhou University, Lanzhou 730000, P.R. China (Email: niedx1993@163.com)}
\author{Weihua Deng}\address{School of Mathematics and Statistics, Gansu Key Laboratory of Applied Mathematics and Complex Systems, Lanzhou University, Lanzhou 730000, P.R. China (Email: dengwh@lzu.edu.cn).}

%
%
\begin{abstract}
In this paper, we provide a framework of designing the local discontinuous Galerkin scheme for integral fractional Laplacian $(-\Delta)^{s}$ with $s\in(0,1)$ in two dimensions. We theoretically prove and numerically verify the numerical stability and convergence of the scheme with the convergence rate no worse than $\mathcal{O}(h^{k+\frac{1}{2}})$. 
\end{abstract}
%
%
\maketitle
\section{Introduction}

Space fractional diffusion equations govern the probability density function of the position of the L\'evy process with the isotropic power law measure of the jump length \cite{Deng.2018BPftFaTFO}. In this paper, we use local discontinuous Galerkin (LDG) method to solve the  two-dimensional space fractional diffusion equation with the homogeneous Dirichlet boundary condition, i.e.,
\begin{equation}\label{equation2D}
 \left\{
 \begin{aligned}
 &\frac{\partial u(\mathbf{x},t)}{\partial t}+(-\Delta)^{s}u(\mathbf{x},t)=f(\mathbf{x},t)\qquad(\mathbf{x},t)\in \Omega\times(0,T],\\
 &u(\mathbf{x},0)=u_{0}(\mathbf{x})\qquad\qquad\qquad\qquad\qquad\mathbf{x}\in \Omega,\\
 &u(\mathbf{x},t)=0\qquad\qquad\qquad\qquad\qquad\qquad(\mathbf{x},t)\in (\mathbb{R}^2\backslash \Omega)\times (0,T],
 \end{aligned}
 \right.
 \end{equation}
where $\Omega\subset \mathbb{R}^{2}$ is a bounded domain; $f(\mathbf{x},t)$ is the source term; $s\in (0,1)$ and the integral fractional Laplacian is defined by \cite{Acosta.2017AFLERoSaFEA,Deng.2018BPftFaTFO}
\begin{equation}\label{fracdefine}
	(-\Delta)^{s}u(\mathbf{x})=c_{2,s}{\rm P.V. }\int_{\mathbb{R}^{2}}\frac{u(\mathbf{x})-u(\mathbf{y})}{|\mathbf{x}-\mathbf{y}|^{2+2s}}d\mathbf{y}
\end{equation}
with $c_{2,s}=\frac{2^{2s}s\Gamma(s+1)}{\pi\Gamma(1-s)}$.

In recent years, the fractional Laplacian has attracted a lot of attentions in both pure and applied mathematical community. Various numerical methods, such as finite element, finite difference, and spectral method, are proposed to solve a wide variety of equations involving integral fractional Laplacian. For example, \cite{Acosta.2017AFLERoSaFEA,Acosta.2017AsFifa2hDpoafL,Acosta.2017Rtahonmft1fL,Acosta.2019Feaffep,Bonito.2019NaotifL,Nie.2020NaftstfFPswtis,Zhang.2018ARBGMftTFL} use finite element method with piecewise linear polynomial to solve equations involving fractional Laplacian with homogeneous Dirichlet boundary condition; \cite{Acosta.2018FeaotnfDpb} discusses the regularity for fractional Poisson equation with nonhomogenous Dirichlet boundary condition and proposes a mixed finite element scheme. \cite{Duo.2018AnaafdmftfLatfPp,Duo.2019AnmftatdifLwa,Huang.2014NMftFLAFDQA} solve the fractional Poisson equation by  finite difference method and an $\mathcal{O}(h^{2})$ convergence rate is obtained. But to our best knowledge, it seems there are no research works on using discontinuous Galerkin method or finite element with polynomial of degree $k$ $(k>1)$ to discretize the fractional Laplacian. The main challenges come from that the variational formulation $((-\Delta)^{s}u,v)$ with $s\in(\frac{1}{2},1)$ is blow-up when $u,v$ are both discontinuous and it is difficult to  generate $((-\Delta)^{s}u,v)$ when $u$, $v$ are polynomials of degree $k$ $(k>1)$.

LDG method is first proposed in \cite{Cockburn.1998TLDGMfTDCDS} and has been widely used to solve integer order partial differential equations  \cite{Castillo.2001OapeefthvotldGmfcdp,Cockburn.2001SotLDGMfEPoCG,Cockburn.2007AAotMDLDGMfCDP,Dong.2009AoaLDGMfLTDFOP,Liu.2019SoECDGMfLHE,Wang.2020LdGmweintdfsndp,Yeganeh.2017SdsdiatfdeualdGm}. For solving the fractional partial differential equations, \cite{Deng.2013LDGmffde} develops LDG scheme for fractional diffusion equation. Then \cite{Xu.2014DGMfFCDE} uses LDG method for solving fractional convection-diffusion equations and \cite{Qiu.2015NdGmffdeo2dwtm} applies LDG method to solve two-dimensional fractional diffusion problem. But it should be noted that the spatial operator is Riemann-Liouville fractional derivative or Riesz fractional derivative for $s\in(\frac{1}{2},1)$ in these works.
In this paper, we propose a suitable split based on its Fourier transform for the integral fractional Laplacian and build a LDG scheme for Eq. \eqref{equation2D} with $s\in(0,1)$. We discuss the stability and convergence of the scheme and obtain a convergence rate $\mathcal{O}(h^{k+\frac{1}{2}})$, where $k$ denotes the degree of the polynomial.  

The rest of the paper is organized as follows. In Section 2, we provide some definitions and the equivalent form of Eq. \eqref{equation2D}. In Section 3, we introduce the construction of the LDG scheme for Eq. \eqref{equation2D} in detail. The stability and error analyses are made in Section 4. In Section 5, we verify the effectiveness of our scheme by some numerical examples. In the last section, we conclude the paper with some discussions. Throughout the paper, $C$ denotes the positive constant, which may differ at different occurrences.

\section{Preliminaries}
In this section, we mainly provide the equivalent form of Eq. \eqref{equation2D}.

Taking Fourier transform for the  integral fractional Laplacian \eqref{fracdefine} and decomposing the symbol $|\boldsymbol{\xi}|^{2s}$, we have
\begin{equation*}
	\begin{aligned}
		\mathcal{F}((-\Delta)^{s}u)(\boldsymbol{\xi})=|\boldsymbol{\xi}|^{2s}\mathcal{F}(u)=-(\mathbf{i}\boldsymbol{\xi})\cdot|\boldsymbol{\xi}|^{2s-2}(\mathbf{i}\boldsymbol{\xi})\mathcal{F}(u),
	\end{aligned}
\end{equation*}
where $\mathcal{F}(u)$ means the Fourier transform of $u$ and $\mathbf{i}^{2}=-1$. Then the property of Fourier transform gives
\begin{equation}\label{eqFLeq}
	\begin{aligned}
	(-\Delta)^{s}u=-\nabla\cdot((-\Delta)^{s-1}\nabla u),
	\end{aligned}
\end{equation}
where $\nabla$ stands for the gradient operator; the definition of $(-\Delta)^{s}$ for $s<0$ is \cite{Vazquez.2012NDwFLO}
\begin{equation}\label{equnelap}
	(-\Delta)^{s}u(\mathbf{x})=c_{d,s}\int_{\mathbb{R}^{d}}\frac{u(\mathbf{y})}{|\mathbf{x}-\mathbf{y}|^{d+2s}}d\mathbf{y}
\end{equation}
with $c_{d,s}=-\frac{2^{2s}s\Gamma(s+\frac{d}{2})}{\pi^{d/2}\Gamma(1-s)}$, and $d$ denotes the dimension of the space with $d=2$ in this paper.

\begin{remark}
	When we take $s=\frac{1}{2}$ in one dimension, the formula \eqref{eqFLeq} does not hold since the $c_{1,-\frac{1}{2}}$ (see \eqref{equnelap}) blows up. But in high dimensions, the formula \eqref{eqFLeq} still holds in the case $s=\frac{1}{2}$.
\end{remark}

According to Eq. \eqref{eqFLeq}, we rewrite Eq. \eqref{equation2D} as
\begin{equation}\label{equationequal}
\left\{
\begin{aligned}
&\frac{\partial u(\mathbf{x},t)}{\partial t}=\nabla\cdot \mathbf{q}+f(\mathbf{x},t)~~\quad\quad\quad\quad\quad(\mathbf{x},t)\in \Omega\times(0,T],\\
&\mathbf{q}=(-\Delta)^{s-1}\mathbf{p}\qquad\quad\quad\quad\qquad\qquad\quad(\mathbf{x},t)\in \Omega\times(0,T],\\
&\mathbf{p}=\nabla u\qquad\qquad\quad\qquad\qquad\qquad\qquad(\mathbf{x},t)\in \Omega\times(0,T],\\
&u(\mathbf{x},0)=u_{0}(\mathbf{x})~\,\quad\quad\quad\quad\quad\qquad\qquad\qquad~\mathbf{x}\in \Omega,\\
&u(\mathbf{x},t)=0\ \,\,\quad\quad\quad\qquad\qquad\qquad\quad\quad(\mathbf{x},t)\in (\mathbb{R}^{2}\backslash\Omega)\times(0,T],\\
\end{aligned}
\right.
\end{equation}
where $\mathbf{q}\in L^{2}(\Omega),~\mathbf{p}\in (H^{1}(\Omega))^{2}$.


\section{Construction of the LDG scheme}

In this section, we provide the spatial semi-discrete scheme for  Eq. \eqref{equation2D} by using the LDG method to discretize the fractional Laplacian.

First, we introduce the computational domain $\Omega_{h}$, which is a well approximation of the physical domain $\Omega$. Here we use the triangular meshes in $\Omega_{h}$. Denote the shape-regular triangular element with diameter $h_j$ as $I_{j}$, satisfying
\begin{equation*}
	\Omega_{h}=\cup_{j=1}^{K}I_{j},\quad h=\max_{1\leq j\leq K}h_{j},
\end{equation*}
where $K$ is the number of elements; and $\Gamma$ consists of all of the boundaries of the elements $I_{j}$, $j=1,2,\ldots,K$. Denote $\Gamma_{\mathbb{I}}$ and $\Gamma_{\mathbb{B}}$ as purely internal edges and external edges of the domain boundaries, respectively; i.e., $\Gamma=\Gamma_{\mathbb{I}}\cup \Gamma_{\mathbb{B}}$.

Then for two continuous functions $f, g$, we define the inner product on the element $I_{j}$ and  over the face of $I_{j}$ as
\begin{equation*}
(f,g)_{I_{j}}=\int_{I_{j}}f(\mathbf{x})g(\mathbf{x})d\mathbf{x},\ (f,g)_{\partial
I_{j}}=\int_{\partial I_{j}}f(s)g(s)ds.
\end{equation*}
Denote the space of $k$-th order polynomials with $k\geq 1$ in two variables on the element $I_{j}$ as $P_{k}(I_{j})$ whose dimension is $N_{k}$, i.e.,
\begin{equation*}
	P_{k}(I_{j})={\rm span}\{l_{i}(\mathbf{x}),i=1,2,\ldots,N_{k}\},
\end{equation*}
where $l_{i}(\mathbf{x})$ denotes the interpolation  basis function with the interpolating points $\{\mathbf{x}_{i}\}_{i=1}^{N_{k}}$. The discontinuous  finite element space $V_{h,k}$ can be defined by
\begin{equation*}
	V_{h,k}=\{v:\Omega_{h}\rightarrow  \mathbb{R}\big|~~v|_{I_{j}}\in P_{k}(I_{j}),~j=1,\ldots,K\}.
\end{equation*}

According to \eqref{equationequal}, $\{u,\mathbf{p},\mathbf{q}\}$ satisfies the following variational form
\begin{equation}\label{equationvar}
	\left\{
	\begin{aligned}
		&\left (\frac{\partial u}{\partial t},v\right )_{I_{i}}=-( \mathbf{q},\nabla v)_{I_{i}}+(\mathbf{n}\cdot\mathbf{q},v)_{\partial I_{i}}+(f,v)_{I_{i}},\\
		&\left (\mathbf{q},\mathbf{w}\right )_{I_{i}}=\left ((-\Delta)^{s-1}\mathbf{p},\mathbf{w}\right )_{I_{i}},\\
		&\left (\mathbf{p},\mathbf{z}\right )_{I_{i}}=-( u, \nabla\cdot\mathbf{z})_{I_{i}}+(u,\mathbf{n}\cdot\mathbf{z})_{\partial I_{i}},\\
		&u(\mathbf{x},0)=u_{0}(\mathbf{x})\quad\quad\quad\quad\quad\quad\quad\quad\qquad\qquad~~\mathbf{x}\in \Omega,\\
		&u(\mathbf{x},t)=0\quad\quad\quad\quad\quad\quad\quad\qquad\qquad\quad(\mathbf{x},t)\in (\mathbb{R}^2\backslash\Omega)\times[0,T],\\
	\end{aligned}
	\right.
\end{equation}
for all test functions $v\in H^{1}(\Omega)$, $\mathbf{w}\in (L^{2}(\Omega))^{2}$, and $\mathbf{z}\in (H^{1}(\Omega))^{2}$. Here $\mathbf{n}$ denotes the outward normal unit vector of $\partial I_{i}$.  Define $\{u_{h},\mathbf{p}_{h},\mathbf{q}_{h}\}$ as the approximation of $\{u,\mathbf{p},\mathbf{q}\}$. Thus the LDG scheme can be written as: find $\{u_{h},\mathbf{p}_{h},\mathbf{q}_{h}\}\in H^{1}(0,T,V_{h,k})\times(L^{2}(0,T,V_{h,k}))^{2}\times(L^{2}(0,T,V_{h,k}))^{2}$ satisfying
\begin{equation}\label{equationdis}
\left\{
\begin{aligned}
&\left (\frac{\partial u_{h}}{\partial t},v_{h}\right )_{I_{i}}=(\nabla\cdot \mathbf{q}_{h}, v_{h})_{I_{i}}-(\mathbf{n}\cdot(\mathbf{q}_{h}-\hat{\mathbf{q}}_{h}),v_{h})_{\partial I_{i}}+(f,v_{h})_{I_{i}},\\
&\left (\mathbf{q}_{h},\mathbf{w}_{h}\right )_{I_{i}}=\left ((-\Delta)^{s-1}\mathbf{p}_{h},\mathbf{w}_{h}\right )_{I_{i}},\\
&\left (\mathbf{p}_{h},\mathbf{z}_{h}\right )_{I_{i}}=(\nabla u_{h}, \mathbf{z}_{h})_{I_{i}}-(u_{h}-\hat{u}_{h},\mathbf{n}\cdot\mathbf{z}_{h})_{\partial I_{i}},\\
&(u_{h}(\mathbf{x},0),v_{h})_{I_{i}}=(u_{0}(\mathbf{x}),v_{h})_{I_{i}},\\
\end{aligned}
\right.
\end{equation}
for all $v_{h}\in V_{h,k}$ and $\mathbf{w}_{h},~\mathbf{z}_{h}\in(V_{h,k})^{2}$, where $\hat{u}_{h}$ and $\hat{\mathbf{q}}_{h}$ are the fluxes determined below.
\begin{remark}
	All the terms in Eq. \eqref{equationdis} can be computed easily except the term $\left ((-\Delta)^{s-1}\mathbf{p}_{h},\mathbf{w}\right )_{I_{i}}$. As for the numerical computation of $\left ((-\Delta)^{s-1}\mathbf{p}_{h},\mathbf{w}\right )_{I_{i}}$, one can refer to \cite{Acosta.2017AsFifa2hDpoafL,Sauter.2011BEM} about the typical integrals appearing in the Boundary Element Method.
\end{remark}
Before specifying the fluxes, let's introduce some notations\cite{Dong.2009AoaLDGMfLTDFOP}. Let $\bar{\Gamma}$ be some fixed face of $I_{j}$. Then for $\bar{\Gamma}\in\partial I_{j}\subset \Gamma$, the average and jump of a scalar function are defined as
\begin{equation*}
	\begin{aligned}
		&\{\!\{u\}\!\}=\frac{u_{ext}+ u_{int}}{2},\quad [u]=\mathbf{n}^{-}u_{ext}+\mathbf{n}^{+} u_{int},\\
	\end{aligned}
\end{equation*}
where `$int$' and `$ext$' mean the interior and exterior information of $I_{j}$ on $\bar{\Gamma}$ and $\mathbf{n}^{+}$ ($\mathbf{n}^{-}$) means the unit normal vector on $\bar{\Gamma}$ pointing exterior (interior) to $I_{j}$; similarly, for the vector function, the average and jump are defined as
\begin{equation*}
	\begin{aligned}
		\{\!\{\mathbf{q}\}\!\}=\frac{\mathbf{q}_{ext}+\mathbf{q}_{int}}{2},\quad [\mathbf{q}]=\mathbf{n}^{-}\cdot\mathbf{q}_{ext}+\mathbf{n}^{+}\cdot \mathbf{q}_{int}.
	\end{aligned}
\end{equation*}
Denote
$u^{\pm}$ and $\mathbf{q}^{\pm}$ as
\begin{equation*}
	u^{\pm}=\{\!\{u\}\!\}\pm\boldsymbol{\beta}\cdot[u],\quad \mathbf{q}^{\pm}=\{\!\{\mathbf{q}\}\!\}\pm\boldsymbol{\beta}\cdot[\mathbf{q}],
\end{equation*}
where $\boldsymbol{\beta}$ is a function on $\Gamma$ satisfying
\begin{equation*}
	\boldsymbol{\beta}\cdot \mathbf{n}=\frac{1}{2}{\rm sign}(\mathbf{1}\cdot \mathbf{n})
\end{equation*}
with the vector $\mathbf{1}=[1,1]^{T}$ and `$\rm sign$' standing for the sign function.
\begin{remark}
	If $\mathbf{1}\cdot \mathbf{n}=0$, one can choose the $\boldsymbol{\beta}$ satisfying
	\begin{equation*}
		\boldsymbol{\beta}\cdot \mathbf{n}=\frac{1}{2}{\rm sign}(\mathbf{1}_{\sigma}\cdot \mathbf{n}),
	\end{equation*}
where $\mathbf{1}_{\sigma}=[1+\sigma,1-\sigma]^{T}$with some $\sigma\neq0$.
\end{remark}

Here, we choose the  alternating direction flux provided in \cite{Dong.2009AoaLDGMfLTDFOP,Cockburn.2007AAotMDLDGMfCDP,Xu.2014DGMfFCDE}, i.e.,
\begin{equation}\label{eqflux1i}
	\begin{aligned}
		&\hat{u}_{h}=u_{h}^{+},\quad\hat{\mathbf{q}}_{h}=\mathbf{q}_{h}^{-}\quad{\rm for}~~ \bar{\Gamma}\in\Gamma_{\mathbb{I}};\\
	&\hat{u}_{h}=0,\quad  \hat{\mathbf{q}}_{h}=\mathbf{q}_{h}^{-}+g^{-}(u),\quad \mathbf{q}_{h}^{-}=\mathbf{q}_{h}^{+}\quad{\rm for}~~ \bar{\Gamma}\in\Gamma_{\mathbb{B}};
	\end{aligned}
\end{equation}
and an alternative choice is
\begin{equation}\label{eqflux2i}
	\begin{aligned}
		&\hat{u}_{h}=u_{h}^{-},\quad\hat{\mathbf{q}}_{h}=\mathbf{q}_{h}^{+}\quad{\rm for}~~ \bar{\Gamma}\in\Gamma_{\mathbb{I}};\\
	&\hat{u}_{h}=0,\quad  \hat{\mathbf{q}}_{h}=\mathbf{q}_{h}^{+}+g^{+}(u),\quad \mathbf{q}_{h}^{+}=\mathbf{q}_{h}^{-}\quad{\rm for}~~ \bar{\Gamma}\in\Gamma_{\mathbb{B}},
	\end{aligned}
\end{equation}
where $g^{\pm}(u)$ satisfies
\begin{equation*}
	 g^{\pm}(u)=\left\{ \begin{aligned}
	&\pm\frac{\vartheta[u]}{h},\quad\mathbf{1}\cdot\mathbf{n}^{\pm}>0,\\
	&0,\qquad {\rm otherwise}
	\end{aligned}\right.
\end{equation*}
with $\vartheta>0$. For convenience, we denote the subspace $\Gamma^{\pm}_{\mathbb{B}}=\{\bar{\Gamma}\in\Gamma_{\mathbb{B}},\mathbf{1}\cdot\mathbf{n}^{\pm}>0\}$.

For the simplicity of the theoretical analysis, we denote
\begin{equation}\label{equdefBi}
	\begin{aligned}
		\mathbf{B}(\boldsymbol{\phi}_{h};\boldsymbol{\psi}_{h})=&\int_{0}^{T}\sum_{i=1}^{K}\Bigg(\left (\frac{\partial u_{h}}{\partial t},v_{h}\right )_{I_{i}}-(\nabla\cdot \mathbf{q}_{h}, v_{h})_{I_{i}}+(\mathbf{n}\cdot(\mathbf{q}_{h}-\hat{\mathbf{q}}_{h}),v_{h})_{\partial I_{i}}\\
		&-\left (\mathbf{q}_{h},\mathbf{w}_{h}\right )_{I_{i}}+\left ((-\Delta)^{s-1}\mathbf{p}_{h},\mathbf{w}_{h}\right )_{I_{i}}\\
		&+\left (\mathbf{p}_{h},\mathbf{z}_{h}\right )_{I_{i}}-(\nabla u_{h}, \mathbf{z}_{h})_{I_{i}}+(u_{h}-\hat{u}_{h},\mathbf{n}\cdot\mathbf{z}_{h})_{\partial I_{i}}\Bigg)dt,
	\end{aligned}
\end{equation}
where $\boldsymbol{\phi}_{h}=\{u_{h},\mathbf{p}_{h},\mathbf{q}_{h}\}$ and $\boldsymbol{\psi}_{h}=\{v_{h},\mathbf{w}_{h},\mathbf{z}_{h}\}$.
Thus, the space semi-discrete scheme can be written as: find $\boldsymbol{\phi}_{h}\in H^{1}(0,T,V_{h,k})\times(L^{2}(0,T,V_{h,k}))^{2}\times(L^{2}(0,T,V_{h,k}))^{2}$ satisfying
\begin{equation*}
	\mathbf{B}(\boldsymbol{\phi}_{h};\boldsymbol{\psi}_{h})=\int_{0}^{T}\sum_{i=1}^{K}(f,v_{h})_{I_{i}}dt,
\end{equation*}
for all $\boldsymbol{\psi}_{h}\in H^{1}(0,T;V_{h,k})\times(L^{2}(0,T;V_{h,k}))^{2}\times(L^{2}(0,T;V_{h,k}))^{2}$.

\section{Stability and error estimates}
In this section, we discuss the stability and convergence of the semi-discrete scheme \eqref{equationdis}.
\subsection{Stability analysis}
Let $\{\bar{u}_{h},\bar{\mathbf{p}}_{h},\bar{\mathbf{q}}_{h}\}\in H^{1}(0,T,V_{h,k})\times(L^{2}(0,T,V_{h,k}))^{2}\times(L^{2}(0,T,V_{h,k}))^{2}$ be the approximations of $\{u_{h},\mathbf{p}_{h},\mathbf{q}_{h}\}$ and $\boldsymbol{\varepsilon}=\{\varepsilon_{u},\boldsymbol{\varepsilon}_{p},\boldsymbol{\varepsilon}_{q}\}=\{u_{h}-\bar{u}_{h},\mathbf{p}_{h}-\bar{\mathbf{p}}_{h},\mathbf{q}_{h}-\bar{\mathbf{q}}_{h}\}$ be the round-off errors, which satisfies
\begin{equation}\label{eqrondoff}
	\mathbf{B}(\boldsymbol{\varepsilon};\boldsymbol{\psi}_{h})=0,
\end{equation}
for all $\boldsymbol{\psi}_{h}\in H^{1}(0,T;V_{h,k})\times(L^{2}(0,T;V_{h,k}))^{2}\times(L^{2}(0,T;V_{h,k}))^{2}$. Then we establish the stability of our scheme \eqref{equationdis}.
\begin{theorem}\label{thmstab1}
	The scheme \eqref{equationdis} with flux \eqref{eqflux1i} is $L^{2}$ stable, and for all $T>0$, we have
		\begin{equation*}
		\begin{aligned}
			\|\varepsilon_{u}(T)\|^{2}_{L^{2}(\Omega_{h})}=\|\varepsilon_{u}(0)\|^{2}_{L^{2}(\Omega_{h})}-2\int_{0}^{T}\left ((-\Delta)^{s-1}\boldsymbol{\varepsilon}_{p},\boldsymbol{\varepsilon}_{p}\right )_{\Omega_{h}}+\left(\varepsilon^{+}_{u},\frac{\vartheta \varepsilon^{+}_{u}}{h} \right)_{\Gamma^{+}_{\mathbb{B}}}dt.
		\end{aligned}
	\end{equation*}
\end{theorem}
\begin{proof}
	According to \eqref{equdefBi}, $\mathbf{B}(\boldsymbol{\varepsilon};\boldsymbol{\varepsilon})$ can be written as
	\begin{equation*}
	\begin{aligned}
	&\mathbf{B}(\boldsymbol{\varepsilon};\boldsymbol{\varepsilon})\\
	=&\int_{0}^{T}\sum_{i=1}^{K}\Bigg(\left (\frac{\partial \varepsilon_{u}}{\partial t},\varepsilon_{u}\right )_{I_{i}}-(\nabla\cdot \boldsymbol{\varepsilon}_{q}, \varepsilon_{u})_{I_{i}}+(\mathbf{n}\cdot(\boldsymbol{\varepsilon}_{q}-\hat{\boldsymbol{\varepsilon}}_{q}),\varepsilon_{u})_{\partial I_{i}}\\
	&-\left (\boldsymbol{\varepsilon}_{q},\boldsymbol{\varepsilon}_{p}\right )_{I_{i}}+\left ((-\Delta)^{s-1}\boldsymbol{\varepsilon}_{p},\boldsymbol{\varepsilon}_{p}\right )_{ I_{i}}\\
	&+\left (\boldsymbol{\varepsilon}_{p},\boldsymbol{\varepsilon}_{q}\right )_{I_{i}}-(\nabla \varepsilon_{u}, \boldsymbol{\varepsilon}_{q})_{I_{i}}+(\varepsilon_{u}-\hat{\varepsilon}_{u},\mathbf{n}\cdot\boldsymbol{\varepsilon}_{q})_{\partial I_{i}}\Bigg)dt=0,\\
	\end{aligned}
	\end{equation*}
	which yields
	\begin{equation*}
	\begin{aligned}
		&\mathbf{B}(\boldsymbol{\varepsilon};\boldsymbol{\varepsilon})\\
		=&\int_{0}^{T}\frac{1}{2} \frac{\partial }{\partial t}\|\varepsilon_{u}\|^{2}_{L^{2}(\Omega_{h})}-(\nabla\cdot \boldsymbol{\varepsilon}_{q}, \varepsilon_{u})_{\Omega_{h}}+\sum_{i=1}^{K}(\mathbf{n}\cdot(\boldsymbol{\varepsilon}_{q}-\hat{\boldsymbol{\varepsilon}}_{q}),\varepsilon_{u})_{\partial I_{i}}\\
	&+\left ((-\Delta)^{s-1}\boldsymbol{\varepsilon}_{p},\boldsymbol{\varepsilon}_{p}\right )_{\Omega_{h}}
	-(\nabla \varepsilon_{u}, \boldsymbol{\varepsilon}_{q})_{\Omega_{h}}+\sum_{i=1}^{K}(\varepsilon_{u}-\hat{\varepsilon}_{u},\mathbf{n}\cdot\boldsymbol{\varepsilon}_{q})_{\partial I_{i}}dt=0.\\
	\end{aligned}
	\end{equation*}
	Using \eqref{eqflux1i} and doing simple calculations show that
	\begin{equation*}
		\begin{aligned}
			&(\nabla\cdot\boldsymbol{\varepsilon}_{q},\varepsilon_{u})_{\Omega_{h}}-\sum_{i=1}^{K}(\mathbf{n}\cdot(\boldsymbol{\varepsilon}_{q}-\hat{\boldsymbol{\varepsilon}}_{q}),\varepsilon_{u})_{\partial I_{i}}+(\nabla \varepsilon_{u}, \boldsymbol{\varepsilon}_{q})_{\Omega_{h}}-\sum_{i=1}^{K}(\varepsilon_{u}-\hat{\varepsilon}_{u},\mathbf{n}\cdot\boldsymbol{\varepsilon}_{q})_{\partial I_{i}}\\
			&\qquad=-\left(\varepsilon^{+}_{u},\frac{\vartheta \varepsilon^{+}_{u}}{h}\right)_{\Gamma^{+}_{\mathbb{B}}}.
		\end{aligned}
	\end{equation*}
Further, by Parseval's equality, we obtain
	\begin{equation*}
		\left ((-\Delta)^{s-1}\boldsymbol{\varepsilon}_{p},\boldsymbol{\varepsilon}_{p}\right )_{\Omega_{h}}\geq 0.
	\end{equation*}
So
	\begin{equation*}
		\begin{aligned}
			&\frac{1}{2} (\|\varepsilon_{u}(T)\|^{2}_{L^{2}(\Omega_{h})}-\|\varepsilon_{u}(0)\|^{2}_{L^{2}(\Omega_{h})})\\
			&\qquad=-\int_{0}^{T}\left ((-\Delta)^{s-1}\boldsymbol{\varepsilon}_{p},\boldsymbol{\varepsilon}_{p}\right )_{\Omega_{h}}+\left(\varepsilon^{+}_{u},\frac{\vartheta \varepsilon^{+}_{u}}{h}\right)_{\Gamma^{+}_{\mathbb{B}}}dt\leq 0,
		\end{aligned}
	\end{equation*}
	which leads to the desired results.
\end{proof}
Similarly, we have
\begin{theorem}\label{thmstab2}
	The scheme \eqref{equationdis} with flux \eqref{eqflux2i} is $L^{2}$ stable, and for all $T>0$, there holds
	\begin{equation*}
		\begin{aligned}
			\|\varepsilon_{u}(T)\|^{2}_{L^{2}(\Omega_{h})}=\|\varepsilon_{u}(0)\|^{2}_{L^{2}(\Omega_{h})}-2\int_{0}^{T}\left( (-\Delta)^{s-1}\boldsymbol{\varepsilon}_{p},\boldsymbol{\varepsilon}_{p}\right )_{\Omega_{h}}+\left(\varepsilon^{-}_{u},\frac{\vartheta \varepsilon^{-}_{u}}{h}\right)_{\Gamma^{-}_{\mathbb{B}}}dt.
		\end{aligned}
	\end{equation*}
\end{theorem}
\subsection{Error estimates}
Introduce $L^{2}$ orthogonal projection operators $\mathcal{P}: L^{2}(\Omega)\rightarrow V_{h,k}$ and $\mathcal{Q}:(L^{2}(\Omega))^{2}\rightarrow (V_{h,k})^{2}$ as, for all the elements $I_{j}$,
\begin{equation}\label{equorth}
	\begin{aligned}
		&(\mathcal{P}u-u,v)_{I_{j}}=0 \qquad \forall v\in P_{k}(I_{j});\\
		&(\mathcal{Q}\mathbf{u}-\mathbf{u},\mathbf{v})_{I_{j}}=0\qquad \forall \mathbf{v}\in (P_{k}(I_{j}))^{2}.\\
	\end{aligned}
\end{equation}

Following \cite{Dong.2009AoaLDGMfLTDFOP,Cockburn.2007AAotMDLDGMfCDP},  we give the definitions of the projections $\mathcal{P}^{+}$, $\mathcal{Q}^{+}$, $\mathcal{P}^{-}$, and $\mathcal{Q}^{-}$.
 Given a scalar function $u\in L^{2}(\Omega)$ and a vector function $\mathbf{u}\in (L^{2}(\Omega))^{2}$, for an arbitrary element $I_{j}$ and an arbitrary edge $\Gamma_{0}\in \partial I_{j}$ that satisfies $\mathbf{1}\cdot\mathbf{n}_{\Gamma_{0}}<0$, then  $\mathcal{P}^{+}u$ and $\mathcal{Q}^{+}\mathbf{u}$ have the following properties, for all the elements $I_{j}$,
\begin{equation}\label{equproPp}
	\begin{aligned}
		(\mathcal{P}^{+}u-u,v)_{I_{j}}=0&\quad \forall v\in P_{k-1}(I_{j})~~~~{\rm if} ~k\geq 1,\\
		(\mathcal{P}^{+}u-u,v)_{\bar{\Gamma}}=0&\quad \forall v\in P_{k}(\bar{\Gamma})~~{\rm and}~~ \forall \bar{\Gamma}\in \partial I_{j},~~\bar{\Gamma}\neq \Gamma_{0},\\
	\end{aligned}
\end{equation}
and
\begin{equation}\label{equproQp}
	\begin{aligned}
		(\mathcal{Q}^{+}\mathbf{u}-\mathbf{u},\mathbf{v})_{I_{j}}=0&\quad \forall \mathbf{v}\in (P_{k-1}(I_{j}))^{2}~~~~{\rm if} ~k\geq 1,\\
		((\mathcal{Q}^{+}\mathbf{u}-\mathbf{u})\cdot \mathbf{n},v)_{\bar{\Gamma}}=0&\quad \forall v\in P_{k}(\bar{\Gamma})~~{\rm and}~~ \forall \bar{\Gamma}\in \partial I_{j},~ ~\bar{\Gamma}\neq \Gamma_{0}.\\
	\end{aligned}
\end{equation}
The projections $\mathcal{P}^{-}$ and $\mathcal{Q}^{-}$ can be similarly defined as follows. Given a scalar function $u\in L^{2}(\Omega)$ and a vector function $\mathbf{u}\in (L^{2}(\Omega))^{2}$, for an arbitrary element $I_{j}$ and an arbitrary edge $\Gamma_{0}\in \partial I_{j}$ that satisfies $\mathbf{1}\cdot\mathbf{n}_{\Gamma_{0}}>0$, then $\mathcal{P}^{-}u$ and $\mathcal{Q}^{-}\mathbf{u}$ satisfy, for all the elements $I_{j}$,
\begin{equation}\label{equproPm}
	\begin{aligned}
		(\mathcal{P}^{-}u-u,v)_{I_{j}}=0&\quad \forall v\in P_{k-1}(I_{j})~~~~{\rm if} ~k\geq 1,\\
		(\mathcal{P}^{-}u-u,v)_{\bar{\Gamma}}=0&\quad \forall v\in P_{k}(\bar{\Gamma})~~~{\rm and}~~~~ \forall \bar{\Gamma}\in \partial I_{j},~ ~\bar{\Gamma}\neq \Gamma_{0},\\
	\end{aligned}
\end{equation}
and
\begin{equation}\label{equproQm}
	\begin{aligned}
		(\mathcal{Q}^{-}\mathbf{u}-\mathbf{u},\mathbf{v})_{I_{j}}=0&\quad \forall \mathbf{v}\in (P_{k-1}(I_{j}))^{2}~~~{\rm if} ~k\geq 1,\\
		((\mathcal{Q}^{-}\mathbf{u}-\mathbf{u})\cdot \mathbf{n},v)_{\bar{\Gamma}}=0&\quad \forall v\in P_{k}(\bar{\Gamma})~~~{\rm and}~~~~ \forall \bar{\Gamma}\in \partial I_{j},~ ~\bar{\Gamma}\neq \Gamma_{0}.\\
	\end{aligned}
\end{equation}

\begin{theorem}\label{thmer1}
	Let $u$ and $u_{h}$ be the solutions of \eqref{equationvar} and \eqref{equationdis} with flux \eqref{eqflux1i}. Then we have
	\begin{equation*}
		\|u-u_{h}\|_{L^{2}(\Omega_{h})}\leq Ch^{k+\frac{1}{2}},
	\end{equation*}
where $k\geq 1$.
\end{theorem}
\begin{proof}
	Introduce
	\begin{equation*}
		\mathbf{e}=\{e_{u},\mathbf{e}_{p},\mathbf{e}_{q}\}=\{u-u_{h},\mathbf{p}-\mathbf{p}_{h},\mathbf{q}-\mathbf{q}_{h}\}.
	\end{equation*}
Combining \eqref{equationvar} and \eqref{equationdis}, we have
\begin{equation*}
	\mathbf{B}(\mathbf{e};\boldsymbol{\psi}_{h})=0
\end{equation*}
for all $\boldsymbol{\psi}_{h}\in H^{1}(0,T;V_{h,k})\times(L^{2}(0,T;V_{h,k}))^{2}\times(L^{2}(0,T;V_{h,k}))^{2}$. Simple calculations lead to
\begin{equation*}
	\begin{aligned}
		&\mathbf{B}(\boldsymbol{\mathcal{R}}\mathbf{e};\boldsymbol{\mathcal{R}}\mathbf{e})\\
		=&\mathbf{B}(\mathbf{e};\boldsymbol{\mathcal{R}}\mathbf{e})+\mathbf{B}(\mathbf{e}^{p};\boldsymbol{\mathcal{R}}\mathbf{e})\\
		=&\mathbf{B}(\mathbf{e}^{p};\boldsymbol{\mathcal{R}}\mathbf{e}),
	\end{aligned}
\end{equation*}
where $\mathbf{e}^{p}=\{e^{p}_{u},\mathbf{e}^{p}_{\mathbf{p}},\mathbf{e}^{p}_{\mathbf{q}}\}=\boldsymbol{\mathcal{R}}\boldsymbol{\phi}-\boldsymbol{\phi}$, $\boldsymbol{\phi}=\{u,\mathbf{p},\mathbf{q}\}$,
$\boldsymbol{\mathcal{R}}=\{\mathcal{P}^{+},\mathcal{Q},\mathcal{Q}^{-}\}$, and $\boldsymbol{\mathcal{R}}\boldsymbol{\phi}=\{\mathcal{P}^{+}u,\mathcal{Q}\mathbf{p},\mathcal{Q}^{-}\mathbf{q}\}$.

As for $\mathbf{B}(\mathbf{e}^{p};\boldsymbol{\mathcal{R}}\mathbf{e})$, there holds
\begin{equation*}
	\mathbf{B}(\mathbf{e}^{p};\boldsymbol{\mathcal{R}}\mathbf{e})=\uppercase\expandafter{\romannumeral1}+\uppercase\expandafter{\romannumeral2}+\uppercase\expandafter{\romannumeral3}+\uppercase\expandafter{\romannumeral4},
\end{equation*}
where
\begin{equation*}
	\begin{aligned}
		&\uppercase\expandafter{\romannumeral1}=\int_{0}^{T}\left (\frac{\partial e^{p}_{u}}{\partial t},\mathcal{P}^{+}e_{u}\right )_{\Omega_{h}}dt,\\
		&\uppercase\expandafter{\romannumeral2}=\int_{0}^{T}(\mathbf{e}^{p}_{\mathbf{q}},\nabla \mathcal{P}^{+}e_{u})_{\Omega_{h}}+(e^{p}_{u},\nabla\cdot \mathcal{Q}^{-} \mathbf{e}_{q})_{\Omega_{h}}+(\mathbf{e}^{p}_{\mathbf{p}},\mathcal{Q}^{-}\mathbf{e}_{\mathbf{q}})_{\Omega_{h}}dt,\\
		&\uppercase\expandafter{\romannumeral3}=-\int_{0}^{T}(\mathbf{e}^{p}_{\mathbf{q}},\mathcal{Q}\mathbf{e}_{\mathbf{p}})_{\Omega_{h}}dt,\\
		&\uppercase\expandafter{\romannumeral4}=-\sum_{i=1}^{K}\int_{0}^{T}(\mathbf{n}\cdot\hat{\mathbf{e}}^{p}_{\mathbf{q}},\mathcal{P}^{+}e_{u})_{\partial I_{i}}+(\hat{e}^{p}_{u},\mathbf{n}\cdot\mathcal{Q}^{-}\mathbf{e}_{\mathbf{q}})_{\partial I_{i}}dt,\\
		&\uppercase\expandafter{\romannumeral5}=\int_{0}^{T}((-\Delta)^{s-1}\mathbf{e}^{p}_{\mathbf{p}},\mathcal{Q}\mathbf{e}_{\mathbf{p}})_{\Omega_{h}}dt.
	\end{aligned}
\end{equation*}
Using the Cauchy-Schwarz inequality and standard approximation theory yields
\begin{equation*}
	\begin{aligned}
		\uppercase\expandafter{\romannumeral1}\leq& \frac{1}{2}\int_{0}^{T}\left \|\frac{\partial e^{p}_{u}}{\partial t}\right \|^{2}_{L^{2}(\Omega_{h})}dt+\frac{1}{2}\int_{0}^{T}\left \|\mathcal{P}^{+}e_{u}\right \|^{2}_{L^{2}(\Omega_{h})}dt\\
		\leq &Ch^{2k+2}+\frac{1}{2}\int_{0}^{T}\left \|\mathcal{P}^{+}e_{u}\right \|^{2}_{L^{2}(\Omega_{h})}dt.
	\end{aligned}
\end{equation*}
From \eqref{equorth}, \eqref{equproPp}, and \eqref{equproQm}, we have $\uppercase\expandafter{\romannumeral2}=0$. As for $\uppercase\expandafter{\romannumeral3}$, taking $\mathbf{v}$ as the approximation of $\mathcal{Q}\mathbf{e}_{\mathbf{p}}$ in $(V_{h,k-1})^{2}$ and using \eqref{equproQm}, we have
\begin{equation*}
	\begin{aligned}
		\uppercase\expandafter{\romannumeral3}\leq& C \left| \int_{0}^{T}(\mathbf{e}^{p}_{\mathbf{q}},\mathcal{Q}\mathbf{e}_{\mathbf{p}}-\mathbf{v})_{\Omega_{h}}dt \right| \\
		\leq& Ch^{2k+1}.
	\end{aligned}
\end{equation*}
 By trace inequality and the Cauchy-Schwarz inequality, there holds

\begin{equation*}
	\begin{aligned}
		\uppercase\expandafter{\romannumeral4}=&-\int_{0}^{T}(\mathbf{n}\cdot\hat{\mathbf{e}}_{\mathbf{q}}^{p},\mathcal{P}^{+}e_{u})_{\Gamma_{\mathbb{B}}}dt\\
		=&-\int_{0}^{T}\left((\mathbf{n}\cdot(\mathcal{Q}^{-}\mathbf{q}-\mathbf{q}),\mathcal{P}^{+}e_{u})_{\Gamma_{\mathbb{B}}^{+}}-\left(\frac{\vartheta e^{p}_{u}}{h},\mathcal{P}^{+}e_{u}\right)_{\Gamma_{\mathbb{B}}^{+}}\right )dt\\
		\leq&Ch^{2k+2}+C\int_{0}^{T}\frac{\mathcal\|\mathcal{P}^{+}e_{u}\|^{2}_{L^{2}(\Gamma^{+}_{\mathbb{B}})}}{h}dt,
	\end{aligned}
\end{equation*}
where we have used the fact $(\frac{\vartheta e^{p}_{u}}{h},\mathcal{P}^{+}e_{u})_{\Gamma_{\mathbb{B}}^{+}}=0$.
As for $\uppercase\expandafter{\romannumeral5}$, Young's inequality and the approximation theory imply that, for $\epsilon>0$,
\begin{equation*}
	\begin{aligned}
		\uppercase\expandafter{\romannumeral5}=&\int_{0}^{T}((-\Delta)^{s-1}\mathbf{e}^{p}_{\mathbf{p}},\mathcal{Q}\mathbf{e}_{\mathbf{p}})_{\Omega_{h}}dt\\
		\leq &C\epsilon^{-1}\int_{0}^{T}((-\Delta)^{s-1}\mathbf{e}^{p}_{\mathbf{p}},\mathbf{e}^{p}_{\mathbf{p}})_{\Omega_{h}}+C\epsilon\int_{0}^{T}((-\Delta)^{s-1}\mathcal{Q}\mathbf{e}_{\mathbf{p}},\mathcal{Q}\mathbf{e}_{\mathbf{p}})_{\Omega_{h}}dt\\
		\leq &C\epsilon^{-1}h^{2k+2}+C\epsilon\int_{0}^{T}((-\Delta)^{s-1}\mathcal{Q}\mathbf{e}_{\mathbf{p}},\mathcal{Q}\mathbf{e}_{\mathbf{p}})_{\Omega_{h}}dt.
	\end{aligned}
\end{equation*}
Similar to the proof of Theorem \ref{thmstab1}, we have
\begin{equation*}
\begin{aligned}
		\mathbf{B}(\boldsymbol{\mathcal{R}}\mathbf{e};\boldsymbol{\mathcal{R}}\mathbf{e})=&\frac{1}{2}(\|\mathcal{P}^{+}e_{u}(T)\|^{2}_{L^{2}(\Omega_{h})}-\|\mathcal{P}^{+}e_{u}(0)\|^{2}_{L^{2}(\Omega_{h})})\\
		&+\int_{0}^{T}((-\Delta)^{s-1}\mathcal{Q}\mathbf{e}_{\mathbf{p}},\mathcal{Q}\mathbf{e}_{\mathbf{p}})_{\Omega_{h}}+\frac{\mathcal\|\mathcal{P}^{+}e_{u}\|^{2}_{L^{2}(\Gamma^{+}_{\mathbb{B}})}}{h}dt.
\end{aligned}
\end{equation*}
Thus by the Gr\"{o}nwall inequality, one has
\begin{equation*}
	\|\mathcal{P}^{+}e_{u}(T)\|^{2}_{L^{2}(\Omega_{h})}+\int_{0}^{T}((-\Delta)^{s-1}\mathcal{Q}\mathbf{e}_{\mathbf{p}},\mathcal{Q}\mathbf{e}_{\mathbf{p}})_{\Omega_{h}}+\frac{\mathcal\|\mathcal{P}^{+}e_{u}\|^{2}_{L^{2}(\Gamma^{+}_{\mathbb{B}})}}{h}dt\leq Ch^{2k+1}.
\end{equation*}
Combining the projection property, the desired result has been obtained.
\end{proof}
Similarly, we have
\begin{theorem}\label{thmer2}
	Let $u$ and $u_{h}$ be the solutions of \eqref{equationvar} and \eqref{equationdis} with flux \eqref{eqflux2i}. Then
	\begin{equation*}
		\|u-u_{h}\|_{L^{2}(\Omega_{h})}\leq Ch^{k+\frac{1}{2}},
	\end{equation*}
where $k\geq 1$.
\end{theorem}

\begin{remark}
	The numerical scheme \eqref{equationdis} can also be applied to solve Eq. \eqref{equation2D} with $s\in(\frac{1}{2},1)$ in one dimension; and the corresponding stability and convergence analyses can be similarly got. As for $s\in(0,\frac{1}{2})$ in one dimension, since $((-\Delta)^{s}u,u)$ may be negative, it seems not easy to get the stability and convergence of numerical scheme \eqref{equationdis}.
\end{remark}

\section{Numerical experiments}
In this section, we present some numerical experiments to validate the above theoretical results. In the temporal direction, we use the backward Euler scheme and take the time step size $\tau$ small enough to ensure the temporal error negligible.
\begin{example}\label{exampl1}
	We take a ball centered at $(0,0)$ with radius $r=1$ as the domain $\Omega$, and the exact solution
	\begin{equation}\label{equrealsol}
		u(\mathbf{x},t)=e^{-t}(1-|\mathbf{x}|^{2})^{p};
	\end{equation}
the source term can be represented by \cite{Dyda.2012FcfpfaeotfL}
\begin{equation}\label{equsourcet}
	f(\mathbf{x},t)=e^{-t}\left (c_{2,s}\frac{\pi\Gamma(-s)\Gamma(p+1)}{\Gamma(p+1-s)}\!_{2}F_{1}(s+1,-p+s;1;|\mathbf{x}|^{2})-(1-|\mathbf{x}|^{2})^{p}\right ).
\end{equation}
Here, we take $p=6$, $T=1$, $\tau=T/20000$, and the degree $k=1,2$. The parameter in the scheme is taken as $\vartheta=5$. When we choose flux \eqref{eqflux1i}, the corresponding results are presented in Tables \ref{tab:p6k1-1} and \ref{tab:p6k2-1}; and when using \eqref{eqflux2i}, the results are given in Tables \ref{tab:p6k1-2} and \ref{tab:p6k2-2}. All the convergence rates are $\mathcal{O}(h^{k+1})$, higher than the predicted one in Theorem \ref{thmer2}.

\begin{table}[htbp]
	\caption{Errors and orders of convergence with $k=1$ and flux \eqref{eqflux1i}}
	\begin{tabular}{cccccc}
		\hline
			$s\backslash h$& 0.6 & 0.3 & 0.15 & 0.1 \\
		\hline
		0.4 & 8.139E-02 & 3.504E-02 & 8.094E-03 & 3.506E-03 \\
		& Rates & 1.2157 & 2.1142 & 2.0637 \\
		0.6 & 7.508E-02 & 2.896E-02 & 6.567E-03 & 2.786E-03 \\
		&Rates  & 1.3744 & 2.1408 & 2.1148 \\
		0.8 & 7.084E-02 & 2.610E-02 & 6.094E-03 & 2.629E-03 \\
		&Rates  & 1.4404 & 2.0987 & 2.0729 \\
		\hline
	\end{tabular}
	\label{tab:p6k1-1}
\end{table}
\begin{table}[htbp]
	\caption{Errors and orders of convergence with $k=2$ and flux \eqref{eqflux1i}}
	\begin{tabular}{cccccc}
		\hline
			$s\backslash h$& 0.6 & 0.3 & 0.15 & 0.1 \\
		\hline
		0.3 & 4.252E-02 & 5.432E-03 & 6.673E-04 & 2.093E-04 \\
		& Rates & 2.9688 & 3.0249 & 2.8592 \\
		0.5 & 3.582E-02 & 3.481E-03 & 3.554E-04 & 9.869E-05 \\
		& Rates & 3.3633 & 3.2918 & 3.1600 \\
		0.7 & 3.268E-02 & 2.946E-03 & 3.073E-04 & 8.585E-05 \\
		&Rates  & 3.4718 & 3.2608 & 3.1455 \\
		\hline
	\end{tabular}
	\label{tab:p6k2-1}
\end{table}
\begin{table}[htbp]
	\begin{center}
		\caption{Errors and orders of convergence with $k=1$ and flux \eqref{eqflux2i}}
		\begin{tabular}{cccccc}
			\hline
			$s\backslash h$ & 0.6 & 0.3&0.15 & 0.1 \\
			\hline
			0.4 & 1.064E-01 & 3.311E-02&8.188E-03 & 3.503E-03 \\
			& Rates & 1.6844&2.0157 & 2.0937 \\
			0.6 & 9.127E-02 & 2.827E-02&6.615E-03 & 2.785E-03 \\
			&Rates  & 1.6911&2.0953 & 2.1334 \\
			0.8 & 7.988E-02 & 2.573E-02&6.129E-03 & 2.629E-03 \\
			&Rates  & 1.6346&2.0697 & 2.0874 \\
			\hline
		\end{tabular}
		\label{tab:p6k1-2}
	\end{center}
\end{table}

\begin{table}[htbp]
	\begin{center}	
		\caption{Errors and orders of convergence with $k=2$ and flux \eqref{eqflux2i}}
		\begin{tabular}{ccccc}
			\hline
			$s\backslash h$& 0.6 & 0.3&0.15 & 0.1 \\
			\hline
			0.3 & 4.641E-02 & 5.136E-03&6.804E-04 & 2.084E-04 \\
			& Rates & 3.1759&2.9161 & 2.9179 \\
			0.5 & 3.872E-02 & 3.331E-03&3.609E-04 & 9.854E-05 \\
			& Rates & 3.5390&3.2062 & 3.2017 \\
			0.7 & 3.407E-02 & 2.889E-03&3.101E-04 & 8.603E-05 \\
			&Rates  & 3.5599&3.2196 & 3.1627 \\
			\hline
		\end{tabular}
		\label{tab:p6k2-2}
	\end{center}
\end{table}

\end{example}
\begin{example}\label{examp22}
	We choose  the same domain $\Omega$ as the one in Example \ref{exampl1}. We take \eqref{equrealsol} and \eqref{equsourcet} with $p=0$, respectively, as the exact solution and source term; from \cite{Acosta.2019Feaffep,Nie.2020NaftstfFPswtis}, it is known that the solution $u\in H^{s+\frac{1}{2}-\epsilon}(\mathbb{R}^{2})$ with $\epsilon>0$ arbitrary small.  Here, we take $T=1$, $\tau=T/20000$, $\vartheta=5$, and $k=1,2$ with flux \eqref{eqflux2i}. The numerical results are shown in Tables \ref{tab:p0k1} and \ref{tab:p0k2}. It can be noted that the errors for $k=2$ are less than the ones for $k=1$, even though the convergence rates are both about $\mathcal{O}(h^{s+\frac{1}{2}})$.
	
	\begin{table}[htbp]
		\begin{center}
				\caption{Errors and orders of convergence in solving Example \ref{examp22} with $k=1$}
			\begin{tabular}{cccccc}
				\hline
				$s\backslash h$& 0.6 & 0.3&0.15 & 0.1 \\
				\hline
				0.3 & 1.725E-01 & 8.510E-02&5.368E-02 & 3.777E-02 \\
				&Rates  & 1.0190&0.6647 & 0.8671 \\
				0.5 & 1.165E-01 & 4.752E-02&2.573E-02 & 1.652E-02 \\
				&Rates  & 1.2944&0.8852 & 1.0919 \\
				0.7 & 7.993E-02 & 2.520E-02&1.130E-02 & 6.475E-03 \\
				& Rates & 1.6654&1.1572 & 1.3730 \\
				\hline
			\end{tabular}
			\label{tab:p0k1}
		\end{center}
	\end{table}
	\begin{table}[htbp]
		\begin{center}
		\caption{Errors and orders of convergence in solving Example \ref{examp22} with $k=2$}
		\begin{tabular}{cccccc}
			\hline
			$s\backslash h$& 0.6 & 0.3 & 0.15 & 0.1 \\
			\hline
			0.3 & 1.204E-01 & 4.407E-02 & 2.167E-02 & 1.453E-02 \\
			&Rates  & 1.4498 & 1.0240 & 0.9867 \\
			0.5 & 1.089E-01 & 2.742E-02 & 1.002E-02 & 6.122E-03 \\
			& Rates & 1.9896 & 1.4519 & 1.2157 \\
			0.7 & 1.011E-01 & 1.868E-02 & 4.882E-03 & 2.625E-03 \\
			&  Rates& 2.4362 & 1.9360 & 1.5303 \\
			\hline
		\end{tabular}
		\label{tab:p0k2}
	\end{center}
	\end{table}
	
\end{example}
\section{Conclusions}

We propose the LDG framework for the integral fractional Laplacian, which has wide interests in pure and applied mathematical community, and a lot of physical and engineering applications. The complete stability and convergence analyses are provided. The numerical experiments are performed with convergence rates $\mathcal{O}(h^{k+1})$,  better the theoretically predicted ones $\mathcal{O}(h^{k+1/2})$, where $k$ is the degree of the polynomial.



\end{document}